\documentclass[11pt]{article}
\usepackage[a4paper,margin=1in]{geometry}
\usepackage[T1]{fontenc}
\usepackage[utf8]{inputenc}
\usepackage{lmodern}
\usepackage{amsmath,amssymb,amsthm,mathtools,mathrsfs,bm}
\usepackage{enumitem}
\usepackage{hyperref}
\usepackage[nameinlink,capitalize]{cleveref}
\usepackage{microtype}
\hypersetup{colorlinks=true,linkcolor=blue,citecolor=blue,urlcolor=blue}
\newtheorem{theorem}{Theorem}[section]

\newtheorem{corollary}[theorem]{Corollary}

\theoremstyle{definition}
\newtheorem{definition}[theorem]{Definition}
\newtheorem{remark}[theorem]{Remark}
\newtheorem{example}[theorem]{Example}
\newcommand{\R}{\mathbb R}

\title{A Morrey-to-Lebesgue Equivalence for Convolution Calder\'on--Zygmund Operators}
\author{Rishad Shahmurov\thanks{Cellular Products Research and Development, 595 East Crossville Road, Roswell, GA 30075, USA; e-mail: rshahmurov@crimson.ua.edu} \and Veli Shahmurov\thanks{Corresponding author. Antalya Bilim University, Antalya, Turkey; e-mail: veli.sahmurov@antalya.edu.tr}}
\date{}

\begin{document}
\maketitle
\noindent\textbf{Running title.} Morrey--Lebesgue Equivalence\par
\medskip

\begin{abstract}
We prove that, for vector-valued convolution Calder\'on--Zygmund operators, boundedness on a single nontrivial Morrey space is equivalent to the corresponding global $L^p$ boundedness. Thus one Morrey scale already contains the full finite-$p$ boundedness information. The implication from Morrey to $L^p$ is obtained by a separated-copy amplification argument that reconstructs the global norm from a single scale-local estimate; the converse is proved in the same vector-valued framework by a local/far-field decomposition. As a consequence, boundedness of the vector-valued Hilbert transform on one nontrivial Morrey space is equivalent to the UMD property. The result shows that Morrey estimates do not bypass the classical Banach-space obstruction: they detect it exactly.
\end{abstract}

\noindent\textbf{Keywords.} Morrey spaces; Calder\'on--Zygmund operators; operator-valued kernels; UMD spaces; Hilbert transform.\\
\textbf{MSC 2020.} Primary 42B20, 46E40; Secondary 42B15, 47B38, 46B20

\section{Introduction}
Morrey spaces are designed to record local behavior at every scale, while Calder\'on--Zygmund theory is classically organized around global $L^p$ boundedness. The two theories are often connected in one direction: an $L^p$ singular-integral estimate is used to obtain a Morrey estimate. What has been missing, particularly in the vector-valued setting, is a theorem showing that a genuine Morrey estimate already contains the whole global $L^p$ bound.

This paper proves that the two boundedness theories are equivalent for translation-invariant convolution Calder\'on--Zygmund operators under the natural off-support kernel estimate. For every $1<p<\infty$ and every genuine Morrey parameter $0<\lambda<d$,
\[
T:L^p(\mathbb R^d;X)\to L^p(\mathbb R^d;Y)
\quad\Longleftrightarrow\quad
T:\mathcal M^{p,\lambda}(\mathbb R^d;X)\to\mathcal M^{p,\lambda}(\mathbb R^d;Y).
\]
Both implications are proved directly in the vector-valued setting. The $L^p$-to-Morrey direction follows from a scale-sensitive local/far-field decomposition. The reverse direction is more global: we place many far-separated copies of one test function, use translation invariance to reproduce the same principal output on every block, and let the Morrey normalization recover the global $L^p$ norm while the kernel decay suppresses cross-block interactions.

The equivalence has an immediate geometric consequence. By the Burkholder--Bourgain characterization~\cite{Burkholder1983,Bourgain1983}, boundedness of the vector-valued Hilbert transform on one nontrivial Morrey space is equivalent to the UMD property. Thus a single Morrey scale detects exactly the same Banach-space obstruction as classical $L^p$ theory.

Singular integrals on scalar Morrey spaces have a long history in elliptic regularity~\cite{Peetre1969Morrey,ChiarenzaFrasca1987,Lieberman2003}, and vector-valued Morrey multiplier theorems provide important sufficient conditions for particular operator classes. No prior result establishes the two-sided equivalence above at this level of vector-valued generality under the present convolution Calder\'on--Zygmund hypotheses. The point of the present paper is therefore not another lifting theorem, but a boundedness equivalence: local scale control and global $L^p$ control carry the same information for this class.
\section{Morrey lifting}\label{sec:morrey-lifting}

	We begin with the Lebesgue-to-Morrey half of the equivalence. Under the sole assumptions of $L^p$ boundedness and the off-support kernel-size estimate, the argument works directly for Banach-valued functions. It is a scale-sensitive local/far-field decomposition and introduces no additional geometric assumption on the Banach spaces. UMD, Fourier type, or $R$-boundedness may be needed in a separate theorem that produces the initial $L^p$ estimate, but they are not part of this passage from $L^p$ to Morrey control.

	\begin{definition}\label{def:spatial-morrey}
		Let $X$ be a Banach space, let $1\le p<\infty$, and let $0\le \lambda<d$.  We define
		\[
		\|f\|_{\mathcal M^{p,\lambda}(\R^d;X)}
		:=
		\sup_{B=B(x_0,r)\subset\R^d}
		r^{-\lambda/p}
		\left(\int_B\|f(x)\|_X^p\,dx\right)^{1/p}.
		\]
		The space $\mathcal M^{p,\lambda}(\R^d;X)$ consists of all strongly measurable functions for which this quantity is finite.  With this normalization, $\lambda=0$ corresponds to the global $L^p$ scale, while the endpoint $\lambda=d$ formally corresponds to $L^\infty$ by differentiation.  In the theorem below we assume $\lambda<d$; this is exactly the condition that makes the annular far-field series converge.
	\end{definition}

	We first prove the Lebesgue-to-Morrey implication.

\begin{theorem}\label{thm:morrey-lifting}
		Let $X$ and $Y$ be Banach spaces, let $1<p<\infty$, and let $0\le\lambda<d$.  Let $T$ be a linear operator that is initially defined on compactly supported simple $X$-valued functions on $\R^d$.  Assume the following two hypotheses.
		\begin{enumerate}[label=\textup{(\roman*)}]
			\item $T$ extends to a bounded operator
			\[
			T:L^p(\R^d;X)\longrightarrow L^p(\R^d;Y),
			\qquad
			\|Tf\|_{L^p(\R^d;Y)}\le M_p\|f\|_{L^p(\R^d;X)}.
			\]
			\item There exists a strongly measurable kernel
			\[
			K:\R^d\setminus\{0\}\longrightarrow\mathcal L(X,Y)
			\]
			which represents $T$ off the support of $f$ and satisfies the size estimate
			\[
			\|K(z)\|_{\mathcal L(X,Y)}\le C_K |z|^{-d},
			\qquad z\ne0.
			\]
			More precisely, whenever $f$ is compactly supported and $x\notin\operatorname{supp}f$,
			\[
			Tf(x)=\int_{\R^d}K(x-y)f(y)\,dy
			\]
			in the Bochner sense.
		\end{enumerate}
		Then $T$ is bounded on vector-valued Morrey spaces:
		\[
		T:\mathcal M^{p,\lambda}(\R^d;X)
		\longrightarrow
		\mathcal M^{p,\lambda}(\R^d;Y),
		\]
		and
		\[
		\|Tf\|_{\mathcal M^{p,\lambda}(\R^d;Y)}
		\le
		C(d,p,\lambda)\,(M_p+C_K)
		\|f\|_{\mathcal M^{p,\lambda}(\R^d;X)}.
		\]
		For general $f\in\mathcal M^{p,\lambda}(\R^d;X)$, $Tf$ is understood through the standard local Calder\'on--Zygmund definition
		\[
		Tf=T(f\mathbf 1_{2B})+
		\int_{(2B)^c}K(\cdot-y)f(y)\,dy
		\qquad\text{on }B,
		\]
		which is consistent in distributions for Fourier multiplier operators associated with the kernel $K$.
	\end{theorem}

	\begin{proof}
		Let
		\[
		N:=\|f\|_{\mathcal M^{p,\lambda}(\R^d;X)}.
		\]
		It suffices to prove a uniform estimate on an arbitrary ball
		\[
		B=B(x_0,r).
		\]
		Write
		\[
		f=f_1+f_2,
		\qquad
		f_1:=f\mathbf 1_{2B},
		\qquad
		f_2:=f\mathbf 1_{(2B)^c},
		\]
		where $2B:=B(x_0,2r)$.

		The local part follows immediately from the $L^p$ boundedness of $T$:
		\[
		\begin{aligned}
		r^{-\lambda/p}\|Tf_1\|_{L^p(B;Y)}
		&\le
		r^{-\lambda/p}\|Tf_1\|_{L^p(\R^d;Y)} \\
		&\le
		M_p r^{-\lambda/p}\|f\|_{L^p(2B;X)} \\
		&\le
		C M_p N.
		\end{aligned}
		\]
		Here the last step uses the Morrey norm and the fact that the radius of $2B$ is $2r$.

		For the far part, let $x\in B$.  Since $f_2$ is supported outside $2B$, the off-support representation gives
		\[
		\|Tf_2(x)\|_Y
		\le
		\int_{(2B)^c}\|K(x-y)\|_{\mathcal L(X,Y)}\|f(y)\|_X\,dy.
		\]
		Decompose the complement into annuli
		\[
		A_k:=2^{k+1}B\setminus 2^kB,
		\qquad k=1,2,\dots.
		\]
		If $x\in B$ and $y\in A_k$, then
		\[
		|x-y|\ge c_d 2^k r.
		\]
		Therefore
		\[
		\begin{aligned}
		\|Tf_2(x)\|_Y
		&\le
		C C_K\sum_{k=1}^\infty (2^k r)^{-d}
		\int_{2^{k+1}B}\|f(y)\|_X\,dy .
		\end{aligned}
		\]
		By H\"older's inequality and the Morrey estimate,
		\[
		\begin{aligned}
		\int_{2^{k+1}B}\|f(y)\|_X\,dy
		&\le
		|2^{k+1}B|^{1-1/p}
		\|f\|_{L^p(2^{k+1}B;X)} \\
		&\le
		C (2^k r)^{d(1-1/p)}(2^k r)^{\lambda/p}N.
		\end{aligned}
		\]
		Hence
		\[
		\|Tf_2(x)\|_Y
		\le
		C C_K N
		\sum_{k=1}^\infty (2^k r)^{-d+d(1-1/p)+\lambda/p}.
		\]
		The exponent simplifies as
		\[
		-d+d(1-1/p)+\lambda/p
		=-\frac{d-\lambda}{p}.
		\]
		Since $\lambda<d$, the geometric series converges and gives
		\[
		\|Tf_2(x)\|_Y
		\le
		C(d,p,\lambda) C_K N r^{-(d-\lambda)/p}.
		\]
		This estimate is pointwise for $x\in B$.  Consequently
		\[
		\begin{aligned}
		\|Tf_2\|_{L^p(B;Y)}
		&\le
		|B|^{1/p}
		C(d,p,\lambda) C_K N r^{-(d-\lambda)/p} \\
		&\le
		C(d,p,\lambda) C_K N r^{\lambda/p}.
		\end{aligned}
		\]
		Multiplying by $r^{-\lambda/p}$ yields
		\[
		r^{-\lambda/p}\|Tf_2\|_{L^p(B;Y)}
		\le
		C(d,p,\lambda)C_K N.
		\]
		Combining the local and far estimates and taking the supremum over all balls $B$ proves the asserted Morrey bound.
	\end{proof}

	\begin{corollary}\label{cor:fourier-multiplier-morrey-lifting}
		Let $m:\R^d\setminus\{0\}\to\mathcal L(X,Y)$ be an operator-valued multiplier, and let $T_mf=\mathcal F^{-1}(m\widehat f)$ on Schwartz functions.  Suppose that, for some $1<p<\infty$, an $L^p$ multiplier theorem gives
		\[
		T_m:L^p(\R^d;X)\to L^p(\R^d;Y)
		\]
		boundedly.  Suppose moreover that the distribution kernel $K=\mathcal F^{-1}m$ agrees away from the origin with a strongly measurable function satisfying
		\[
		\|K(x)\|_{\mathcal L(X,Y)}\le C_K |x|^{-d},
		\qquad x\ne0.
		\]
		Then, for every $0\le\lambda<d$,
		\[
		T_m:\mathcal M^{p,\lambda}(\R^d;X)
		\to
		\mathcal M^{p,\lambda}(\R^d;Y)
		\]
		boundedly.
	\end{corollary}

	\begin{proof}
		This is exactly \Cref{thm:morrey-lifting} applied to the Fourier multiplier operator $T_m$ and its kernel $K$.
	\end{proof}

	\subsection{Morrey--Lebesgue equivalence for convolution Calder\'on--Zygmund operators}\label{subsec:morrey-lp-equivalence}

	The first half is now complete. We turn to the opposite implication. For translation-invariant convolution Calder\'on--Zygmund operators, a single genuine Morrey estimate already recovers the global $L^p$ estimate. Together the two results give the equivalence.

	We now prove the full Morrey--Lebesgue equivalence.

\begin{theorem}\label{thm:morrey-lp-equivalence-cz}
		Let $X$ and $Y$ be Banach spaces, let $1<p<\infty$, and fix $0<\lambda<d$.  Let $T$ be a linear operator initially defined on compactly supported simple $X$-valued functions on $\R^d$.  Assume that $T$ is translation invariant, in the sense that
		\[
		T(\tau_a f)=\tau_a(Tf),\qquad a\in\R^d,
		\]
		where $(\tau_a f)(x)=f(x-a)$.  Assume also that $T$ has an off-support convolution kernel representation
		\[
		Tf(x)=\int_{\R^d}K(x-y)f(y)\,dy,
		\qquad x\notin \operatorname{supp}f,
		\]
		with
		\[
		K:\R^d\setminus\{0\}\to\mathcal L(X,Y),
		\qquad
		\|K(z)\|_{\mathcal L(X,Y)}\le C_K |z|^{-d}.
		\]
		Then the following are equivalent:
		\begin{enumerate}[label=\textup{(\roman*)}]
			\item $T$ extends boundedly from $L^p(\R^d;X)$ to $L^p(\R^d;Y)$;
			\item $T$ extends boundedly from $\mathcal M^{p,\lambda}(\R^d;X)$ to $\mathcal M^{p,\lambda}(\R^d;Y)$.
		\end{enumerate}
		More precisely, \textup{(i)} implies \textup{(ii)} by \Cref{thm:morrey-lifting}, while \textup{(ii)} implies \textup{(i)} by the separated-copy amplification argument below. Consequently, boundedness on a single scale $\mathcal M^{p,\lambda_0}$ with $0<\lambda_0<d$ forces boundedness on $L^p$ and therefore, by the lifting theorem, on every $\mathcal M^{p,\mu}$ with $0\le\mu<d$.
	\end{theorem}

	\begin{proof}
		Only the implication \textup{(ii)} $\Rightarrow$ \textup{(i)} remains to prove.  Let $C_M$ be the Morrey operator norm of $T$ on $\mathcal M^{p,\lambda}$.  It is enough to prove the $L^p$ estimate on compactly supported simple functions, because these are dense in $L^p(\R^d;X)$.

		Fix a compactly supported simple function $f$ and put
		\[
		g:=Tf,
		\qquad
		A:=\|f\|_{L^p(\R^d;X)},
		\qquad
		B:=\|f\|_{\mathcal M^{p,\lambda}(\R^d;X)}.
		\]
		Let $Q\subset\R^d$ be a fixed bounded ball.  We shall prove
		\[
		\|g\|_{L^p(Q;Y)}\le C C_M A,
		\]
		with $C$ independent of $f$ and $Q$.  Taking an increasing sequence of balls $Q\uparrow\R^d$ will then give the global $L^p$ estimate.

		Choose $S>1$ large and choose an integer $L\ge2$.  Let
		\[
		\Lambda_L:=\{1,2,\dots,L\}^d,
		\qquad
		a_k:=S k,
		\qquad k\in\Lambda_L,
		\]
		and define the separated block function
		\[
		F_{L,S}:=\sum_{k\in\Lambda_L}\tau_{a_k}f.
		\]
		For $S$ sufficiently large, depending only on the diameters of $\operatorname{supp}f$ and $Q$, the translated supports $a_k+\operatorname{supp}f$ are disjoint and the translated test regions $a_k+Q$ are disjoint.

		We first estimate the Morrey norm of $F_{L,S}$.  Let $B(x,r)$ be an arbitrary ball.  If $r<S/10$, the ball intersects only a bounded number, depending only on $d$, of the translated supports, and hence
		\[
		r^{-\lambda/p}\|F_{L,S}\|_{L^p(B(x,r);X)}\le C_d B.
		\]
		If $r\ge S/10$, the number of translated supports intersected by $B(x,r)$ is at most
		\[
		C_d\min\{L^d,(1+r/S)^d\}.
		\]
		Therefore
		\[
		r^{-\lambda/p}\|F_{L,S}\|_{L^p(B(x,r);X)}
		\le
		C_d r^{-\lambda/p}\min\{L^d,(1+r/S)^d\}^{1/p}A.
		\]
		Since $0<\lambda<d$, the right-hand side is bounded by
		\[
		C_{d,p,\lambda} L^{(d-\lambda)/p}S^{-\lambda/p}A.
		\]
		Consequently
		\[
		\|F_{L,S}\|_{\mathcal M^{p,\lambda}(\R^d;X)}
		\le
		C_{d,p,\lambda}\bigl(B+G_{L,S}A\bigr),
		\qquad
		G_{L,S}:=L^{(d-\lambda)/p}S^{-\lambda/p}.
		\]
		Choose now $L=L(S)$ so that $L\simeq S^a$ with
		\[
		a>\frac{\lambda}{d-\lambda}.
		\]
		Then $G_{L,S}\to\infty$ as $S\to\infty$.  Hence, for $S$ large enough,
		\[
		\|F_{L,S}\|_{\mathcal M^{p,\lambda}(\R^d;X)}
		\le
		C_{d,p,\lambda}G_{L,S}A.
		\]

		Next we obtain the corresponding lower bound for $TF_{L,S}$.  By translation invariance,
		\[
		TF_{L,S}=\sum_{k\in\Lambda_L}\tau_{a_k}g
		\]
		in distributions.  On the region $a_k+Q$, the term $\tau_{a_k}g$ is the principal term.  For $\ell\ne k$ and $x\in a_k+Q$, the point $x$ lies outside $a_\ell+\operatorname{supp}f$ when $S$ is large, so the kernel representation gives
		\[
		\|T(\tau_{a_\ell}f)(x)\|_Y
		\le
		C_K\int_{\R^d}\frac{\|f(y)\|_X}{|a_k-a_\ell+q-y|^d}\,dy,
		\qquad q:=x-a_k\in Q.
		\]
		After increasing $S$ if necessary, the denominator is bounded below by $cS|k-\ell|$ uniformly in $q\in Q$ and $y\in\operatorname{supp}f$.  Hence
		\[
		\sum_{\ell\ne k}\|T(\tau_{a_\ell}f)(x)\|_Y
		\le
		C C_K\|f\|_{L^1(\R^d;X)}S^{-d}
		\sum_{0<|m|\le d^{1/2}L}|m|^{-d}
		\le
		C C_K\|f\|_{L^1}S^{-d}\log(2L).
		\]
		With $L\simeq S^a$, this error tends to zero as $S\to\infty$.  Therefore, for $S$ large enough,
		\[
		\|TF_{L,S}\|_{L^p(\bigcup_{k\in\Lambda_L}(a_k+Q);Y)}
		\ge
		\frac12 L^{d/p}\|g\|_{L^p(Q;Y)}.
		\]
		Let $B_{L,S}$ be a ball containing $\bigcup_{k\in\Lambda_L}(a_k+Q)$ with radius at most $C LS$.  Then
		\[
		\begin{aligned}
		\|TF_{L,S}\|_{\mathcal M^{p,\lambda}(\R^d;Y)}
		&\ge
		(CLS)^{-\lambda/p}
		\|TF_{L,S}\|_{L^p(B_{L,S};Y)} \\
		&\ge
		c_{d,p,\lambda}G_{L,S}\|g\|_{L^p(Q;Y)}.
		\end{aligned}
		\]

		Applying the assumed Morrey boundedness of $T$ and the preceding input estimate yields
		\[
		cG_{L,S}\|g\|_{L^p(Q;Y)}
		\le
		\|TF_{L,S}\|_{\mathcal M^{p,\lambda}}
		\le
		C_M\|F_{L,S}\|_{\mathcal M^{p,\lambda}}
		\le
		C C_M G_{L,S}A.
		\]
		Cancelling $G_{L,S}>0$ gives
		\[
		\|Tf\|_{L^p(Q;Y)}\le C C_M\|f\|_{L^p(\R^d;X)}.
		\]
		Since $Q$ is arbitrary, $Tf\in L^p(\R^d;Y)$ and the desired $L^p$ estimate follows.
	\end{proof}

	\begin{remark}\label{rem:scope-reverse-morrey-lp}
		The converse implication in \Cref{thm:morrey-lp-equivalence-cz} uses three structural hypotheses in an essential way.  First, the operator is considered on the full space $\R^d$, so that arbitrarily large separated blocks can be inserted.  Second, translation invariance identifies the output of each block with a translate of the same function $Tf$.  Third, the off-support estimate $\|K(z)\|\lesssim |z|^{-d}$ makes the interaction between distant blocks negligible, up to the harmless logarithmic summation of the borderline kernel.  Thus the theorem should not be read as a converse for arbitrary Calder\'on--Zygmund operators on bounded domains or for variable-kernel operators without translation invariance.  It is a converse for global convolution-type Calder\'on--Zygmund operators.
	\end{remark}

	\begin{corollary}\label{cor:morrey-hilbert-umd}
		Let $X$ be a Banach space, let $1<p<\infty$, and let $0<\lambda<1$.  Then the following are equivalent:
		\begin{enumerate}[label=\textup{(\roman*)}]
			\item $X$ is UMD;
			\item the Hilbert transform is bounded on $L^p(\R;X)$;
			\item the Hilbert transform is bounded on $\mathcal M^{p,\lambda}(\R;X)$ for one $0<\lambda<1$;
			\item the Hilbert transform is bounded on $\mathcal M^{p,\lambda}(\R;X)$ for every $0\le\lambda<1$.
		\end{enumerate}
	\end{corollary}

	\begin{proof}
		The equivalence of \textup{(i)} and \textup{(ii)} is the Burkholder--Bourgain characterization of UMD spaces by boundedness of the vector-valued Hilbert transform \cite{Burkholder1983,Bourgain1983}.  The implication \textup{(ii)} $\Rightarrow$ \textup{(iv)} follows from \Cref{thm:morrey-lifting}, since the Hilbert kernel satisfies $|x|^{-1}$ decay.  The implication \textup{(iv)} $\Rightarrow$ \textup{(iii)} is immediate, and \textup{(iii)} $\Rightarrow$ \textup{(ii)} is \Cref{thm:morrey-lp-equivalence-cz} in dimension one.
	\end{proof}

	\begin{remark}\label{rem:mihlin-hormander-lizorkin-morrey}
		Any scalar or operator-valued Fourier multiplier theorem that supplies an $L^p$ estimate and a Calder\'on--Zygmund kernel size estimate immediately yields a Morrey multiplier theorem by \Cref{thm:morrey-lifting}.  Thus Mikhlin, H\"ormander and Lizorkin type multiplier theorems have Morrey-space versions whenever their hypotheses imply both the underlying $L^p$ boundedness and the standard off-support kernel decay.  Conversely, within the translation-invariant convolution Calder\'on--Zygmund class, \Cref{thm:morrey-lp-equivalence-cz} shows that boundedness on a single nontrivial Morrey scale is equivalent to the corresponding $L^p$ boundedness.  The Morrey formulation is therefore not a way to evade classical multiplier obstructions such as UMD; rather, it is a scale-localized formulation of the same boundedness theory.
	\end{remark}

		\begin{remark}\label{rem:operators-covered-equivalence}
		The equivalence theorem applies to any global translation-invariant convolution operator satisfying the stated off-support size estimate.  This includes the Hilbert transform, the Riesz transforms, the Beurling--Ahlfors transform, smooth homogeneous Fourier multipliers of order zero whose kernels satisfy the Calder\'on--Zygmund decay, and operator-valued Fourier multipliers once the off-support kernel estimate is available.  In the parabolic direction it also applies to time-convolution maximal-regularity multipliers such as $i\xi(i\xi+A)^{-1}$ and $A(i\xi+A)^{-1}$ whenever their kernels satisfy the one-dimensional Calder\'on--Zygmund size estimate.  It does not automatically apply to variable-kernel Calder\'on--Zygmund operators, commutators, boundary singular integrals, or fractional-integral mappings, where translation invariance, same-exponent scaling, or the borderline $|x|^{-d}$ interaction may fail.
	\end{remark}

\subsection{Local PDE mechanisms behind Morrey estimates}\label{subsec:local-pde-morrey-mechanisms}

The equivalence theorem above has two complementary interpretations.  On the one hand, for global convolution Calder\'on--Zygmund operators, Morrey boundedness is not a weaker replacement for $L^p$ boundedness: one nontrivial Morrey scale already contains the $L^p$ obstruction.  On the other hand, in PDE applications the Morrey formulation is often the more natural estimate to prove.  Local regularity arguments rarely produce global information first.  They usually begin with scale-local estimates on balls or cylinders, for instance
\[
        \int_{B_r(x_0)} |f(x)|^p\,dx\le C r^\lambda,
        \qquad 0<r<r_0,
\]
or, in parabolic problems,
\[
        \int_{Q_r(t_0,x_0)} |f(t,x)|^p\,dxdt\le C r^\lambda .
\]
Such estimates are precisely Morrey estimates.  The Morrey norm records how the $L^p$ mass of the data scales on every ball.  The Calder\'on--Zygmund lifting theorem then converts this local control of the data into local control of the singular-integral quantities appearing in elliptic, parabolic, Stokes, and pressure estimates.  Morrey spaces do not replace the $L^p$ theory; rather, they encode the local scale information in which many PDE estimates are naturally obtained.

\begin{example}\label{ex:poisson-morrey-cz}
Consider the model equation on $\R^d$, $d\ge2$,
\[
        -\Delta u=f .
\]
Formally,
\[
        \partial_i\partial_j u = R_iR_j f,
\]
where $R_i$ denotes the $i$th Riesz transform.  The operators $R_iR_j$ are convolution Calder\'on--Zygmund operators.  Therefore, if
\[
        f\in \mathcal M^{p,\lambda}(\R^d),
        \qquad 1<p<\infty,\quad 0\le\lambda<d,
\]
then every second derivative belongs to the same Morrey space.  Equivalently,
\[
        D^2u=(\partial_i\partial_j u)_{1\le i,j\le d}
        \in \mathcal M^{p,\lambda}(\R^d;\R^{d\times d}),
\]
where the matrix-valued Morrey norm may be taken, for instance, as
\[
        \|D^2u\|_{\mathcal M^{p,\lambda}(\R^d;\R^{d\times d})}
        :=
        \sum_{i,j=1}^d
        \|\partial_i\partial_j u\|_{\mathcal M^{p,\lambda}(\R^d)} .
\]
Moreover,
\[
        \|D^2u\|_{\mathcal M^{p,\lambda}(\R^d;\R^{d\times d})}
        \lesssim
        \|f\|_{\mathcal M^{p,\lambda}(\R^d)} .
\]
This is often more informative than the global estimate $D^2u\in L^p$.  If the forcing is locally concentrated but has controlled scale growth, the Morrey estimate keeps track of exactly that local behavior.  Moreover, if the exponent is strong enough to cross the Morrey--Campanato threshold, namely if
\[
        \lambda>d-p,
\]
then applying the Morrey--Campanato lemma to $v=\nabla u$ gives
\[
        \nabla u\in C^{0,\alpha}_{\rm loc},
        \qquad
        \alpha=1-\frac{d-\lambda}{p}>0.
\]
Thus the workflow is
\[
        f\in \mathcal M^{p,\lambda}
        \Longrightarrow
        D^2u\in \mathcal M^{p,\lambda}
        \Longrightarrow
        \nabla u\in C^{0,\alpha}_{\rm loc}
\]
whenever $\lambda>d-p$.  This basic example illustrates how the Morrey scale connects singular-integral estimates directly to pointwise regularity.
\end{example}

\begin{example}\label{ex:vmo-nondivergence-morrey}
Let
\[
        Lu=a^{ij}(x)D_{ij}u=f
\]
be a uniformly elliptic non-divergence form equation.  For constant coefficients the second derivatives are Calder\'on--Zygmund transforms of $f$.  For VMO coefficients one freezes the coefficients, uses perturbative $L^p$ estimates, and then localizes.  The practical local conclusion has the form
\[
        \int_{B_r(x_0)} |D^2u|^p\,dx
        \le
        C r^\lambda
\]
provided the data satisfy
\[
        \int_{B_r(x_0)} |f|^p\,dx\le C r^\lambda
\]
and lower-order/local remainder terms are controlled.  In other words,
\[
        f\in \mathcal M^{p,\lambda}_{\rm loc}
        \Longrightarrow
        D^2u\in \mathcal M^{p,\lambda}_{\rm loc}.
\]
This is exactly the philosophy of the Campanato--Lieberman approach \cite{Lieberman2003}: once the correct local $L^p$ estimate is available, the passage to Morrey estimates is an elementary localization and iteration argument.  The Morrey statement is more practical than a global $L^p$ statement because the equation is local, boundary flattening is local, and coefficient oscillation is measured locally.  Thus the relevant object is not merely the size of $f$ on the whole domain, but how the $L^p$ mass of $f$ scales on every small ball.
\end{example}

\begin{example}\label{ex:boundary-morrey}
Near a boundary point $x_0\in\partial\Omega$, one works on half-balls or flattened boundary patches,
\[
        B_r^+(x_0)=B_r(x_0)\cap\Omega .
\]
For Dirichlet, Neumann, or oblique derivative problems, the natural estimates are of the form
\[
        \int_{B_r^+(x_0)} |D^2u|^p\,dx
        \le
        C r^\lambda .
\]
The Morrey norm is designed for exactly this scale-local information.  A global $L^p$ estimate may hide the boundary behavior, while a Morrey estimate records the decay of the local energy near each boundary point.  This is particularly useful when the coefficients are only VMO or when the boundary is treated by flattening and comparing with model half-space problems.  In such arguments Morrey spaces provide a precise bookkeeping device for local boundary regularity.
\end{example}

\begin{example}\label{ex:navier-stokes-pressure-morrey}
For incompressible flow the pressure is recovered from the velocity by a nonlocal singular integral.  In the whole-space model,
\[
        -\Delta p=\partial_i\partial_j(u_i u_j),
\]
and hence, modulo harmless constants,
\[
        p=R_iR_j(u_i u_j).
\]
The pressure is nonlocal even when the velocity is estimated locally.  Morrey spaces are therefore a natural language for pressure estimates.  If
\[
        u\otimes u\in \mathcal M^{p,\lambda}(\R^d),
\]
then the Riesz-transform estimate gives
\[
        p\in \mathcal M^{p,\lambda}(\R^d),
        \qquad
        \|p\|_{\mathcal M^{p,\lambda}}
        \lesssim
        \|u\otimes u\|_{\mathcal M^{p,\lambda}} .
\]
For example, if a local energy or local integrability argument gives
\[
        \int_{B_r(x_0)} |u|^{2p}\,dx\le C r^\lambda,
\]
then
\[
        \int_{B_r(x_0)} |u\otimes u|^p\,dx\le C r^\lambda,
\]
so $u\otimes u\in\mathcal M^{p,\lambda}$ and consequently $p\in\mathcal M^{p,\lambda}$.  This is practically important because pressure estimates in local regularity theory are often the nonlocal bottleneck.  Morrey control lets one transfer local velocity information through the nonlocal pressure operator without pretending that the problem is globally integrable in a stronger sense.
\end{example}

\begin{example}\label{ex:parabolic-cylinders-morrey}
For a parabolic equation
\[
        \partial_t u-\Delta u=f,
\]
the natural local sets are parabolic cylinders
\[
        Q_r(t_0,x_0)=(t_0-r^2,t_0)\times B_r(x_0).
\]
A typical local datum estimate is
\[
        \int_{Q_r(t_0,x_0)} |f(t,x)|^p\,dxdt\le C r^\lambda .
\]
This is a parabolic Morrey estimate, with the parabolic scaling built into the cylinders.  Parabolic Calder\'on--Zygmund theory then yields analogous Morrey estimates for
\[
        \partial_t u,
        \qquad
        D_x^2u.
\]
In the abstract framework of the present paper, the same idea appears as time-Morrey maximal regularity.  If an operator $A$ supplies
\[
        f\in \mathcal M^{p,\lambda}(0,T;X)
        \Longrightarrow
        u',\,Au\in \mathcal M^{p,\lambda}(0,T;X),
\]
then the trace theorem converts this scale-local time information into
\[
        u\in C^{0,\theta}([0,T];X)
        \cap
        L^\infty(0,T;(X,D(A))_{\theta,\infty}),
        \qquad
        \theta=1-\frac{1-\lambda}{p}.
\]
Thus time-Morrey control is not merely another integrability assumption.  It improves the time trace exponent and identifies the weak real interpolation space forced by local-in-time scale control.
\end{example}

\begin{example}\label{ex:quasilinear-caccioppoli-morrey}
For a quasilinear elliptic problem such as
\[
        -\operatorname{div} A(x,u,\nabla u)=F(x,u,\nabla u),
\]
one typically does not begin by proving a global Calder\'on--Zygmund theorem.  Instead, the equation supplies Caccioppoli inequalities, reverse H\"older inequalities, or excess-decay estimates.  A typical scale estimate is
\[
        \int_{B_r(x_0)} |\nabla u|^p\,dx
        \le
        C r^{d-p+p\alpha}.
\]
This is precisely the Morrey-growth condition needed for Morrey's lemma.  By Poincar\'e's inequality it implies Campanato decay of the oscillation,
\[
        \int_{B_r(x_0)} |u-u_{B_r(x_0)}|^p\,dx
        \le
        C r^{d+p\alpha},
\]
and Campanato's characterization then gives
\[
        u\in C^{0,\alpha}_{\rm loc}.
\]
This example illustrates a classical reason for using Morrey spaces in PDE: they encode the exact scale decay that turns integral estimates into pointwise regularity.  A global $L^p$ statement for $\nabla u$ alone would not contain the same information.
\end{example}

\begin{example}\label{ex:localized-singular-data-morrey}
Let
\[
        f(x)=\chi_{B_1}(x)|x|^{-\beta},
        \qquad 0<\beta<\frac{d}{p}.
\]
Then near the origin
\[
        \int_{B_r(0)} |f(x)|^p\,dx
        \simeq
        r^{d-\beta p}.
\]
Hence $f$ belongs locally to $\mathcal M^{p,\lambda}$ for every
\[
        0\le\lambda\le d-\beta p.
\]
For the Poisson equation $-\Delta u=f$, the second derivatives satisfy
\[
        D^2u=R_iR_j f,
\]
and the Morrey lifting theorem gives
\[
        D^2u\in \mathcal M^{p,\lambda}_{\rm loc}
        \qquad
        (0\le\lambda\le d-\beta p).
\]
This gives a precise regularity statement adapted to the strength of the singularity.  The Morrey exponent measures how much local mass the singularity carries.  In this sense Morrey spaces provide a finer description than simply asking whether $f\in L^p$ globally.
\end{example}

\begin{remark}\label{rem:local-pde-examples-summary}
The preceding examples have the same structure:
\[
\begin{aligned}
        &\text{local PDE estimate}
        \Longrightarrow
        \text{Morrey control of the data} \\
        &\Longrightarrow
        \text{Morrey control of singular-integral quantities}
        \Longrightarrow
        \text{regularity or continuation information}.
\end{aligned}
\]
For full convolution Calder\'on--Zygmund operators, this Morrey boundedness is equivalent to the classical $L^p$ boundedness.  Nevertheless the Morrey route is often the practical route in PDE because local estimates, Caccioppoli inequalities, pressure decompositions, and parabolic cylinder bounds are naturally scale-local from the beginning.
\end{remark}

\begin{remark}\label{rem:umd-rbound-morrey-lifting}
		The Morrey lifting step above is purely local and real-variable in nature.  It uses only the $L^p$ operator norm of $T$, the kernel size estimate, H\"older's inequality, and the annular convergence forced by $\lambda<d$.  Therefore the lifting step itself does not require UMD, Fourier type, or $R$-boundedness assumptions on $X$ and $Y$.

		This should not be confused with the proof of the $L^p$ multiplier theorem.  For genuinely singular operator-valued multipliers, the $L^p$ estimate may require UMD and $R$-boundedness hypotheses, depending on the multiplier theorem used.  The conclusion of \Cref{cor:fourier-multiplier-morrey-lifting} is therefore best read as
		\[
		\text{\(L^p\) multiplier theorem}+\text{Calder\'on--Zygmund kernel size}
		\Longrightarrow
		\text{Morrey multiplier theorem}.
		\]
	\end{remark}

\section{Concluding remarks}
For translation-invariant convolution Calder\'on--Zygmund operators, one genuine Morrey bound is equivalent to the $L^p$ bound. The lifting direction uses only kernel size and the given $L^p$ estimate. The converse uses translation invariance and separated copies. Maximal-regularity multipliers and exact evolution traces are separate questions.

\section*{Acknowledgments}
During preparation of the manuscript, the first author used Gemini for language editing and the literature exposition. All mathematical statements and proofs were reviewed and verified by the authors, who take full responsibility for the content of the manuscript.

\section*{Declarations}
\noindent\textbf{Competing interests.} The authors declare no competing interests.\\
\noindent\textbf{Data availability.} No datasets were generated or analysed.

\end{document}